\documentclass[leqno,11pt]{amsart}

\usepackage{amssymb, amsmath}
\usepackage{amssymb,amsfonts}
\usepackage{amsmath,latexsym}
\usepackage{pdfsync}
\usepackage{bbm}

\newtheorem{theo}{Theorem}[section]

\newtheorem{lemme}{Lemma}[section]
\newtheorem{defin}{Definition}[section]

\numberwithin{equation}{section}

\makeatletter
\def\sommaire{\@restonecolfalse\if@twocolumn\@restonecoltrue\onecolumn
\fi\chapter*{Sommaire\@mkboth{SOMMAIRE}{SOMMAIRE}}
  \@starttoc{toc}\if@restonecol\twocolumn\fi}
\makeatother

\makeatletter
\def\thebibliographie#1{\chapter*{Bibliographie\@mkboth
  {BIBLIOGRAPHIE}{BIBLIOGRAPHIE}}\list
  {[\arabic{enumi}]}{\settowidth\labelwidth{[#1]}\leftmargin\labelwidth
  \advance\leftmargin\labelsep
  \usecounter{enumi}}
  \def\newblock{\hskip .11em plus .33em minus .07em}
  \sloppy\clubpenalty4000\widowpenalty4000
  \sfcode`\.=1000\relax}

\makeatother

\makeatletter
\def\references#1{\section*{R\'ef\'erences\@mkboth
  {R\'EF\'ERENCES}{R\'EF\'ERENCES}}\list
  {[\arabic{enumi}]}{\settowidth\labelwidth{[#1]}\leftmargin\labelwidth
  \advance\leftmargin\labelsep
  \usecounter{enumi}}
  \def\newblock{\hskip .11em plus .33em minus .07em}
  \sloppy\clubpenalty4000\widowpenalty4000
  \sfcode`\.=1000\relax}

\makeatother

\def\refer#1{~\ref{#1}}
\def\refeq#1{~(\ref{#1})}
\def\ccite#1{~\cite{#1}}

\def\longformule#1#2{
\displaylines{
\qquad{#1}
\hfill\cr
\hfill {#2}
\qquad\cr
}
}
\def\inte#1{
\displaystyle\mathop{#1\kern0pt}^\circ
}

\newcommand{\beq}{\begin{equation}}
\newcommand{\eeq}{\end{equation}}
\newcommand{\ben}{\begin{eqnarray}}
\newcommand{\een}{\end{eqnarray}}
\newcommand{\beno}{\begin{eqnarray*}}
\newcommand{\eeno}{\end{eqnarray*}}

\let\al=\alpha

\let\g=\gamma

\let\e=\varepsilon

\let\lam=\lambda

\let\s=\sigma

\let\D=\Delta
\let\Lam=\Lambda

\let\wt=\widetilde
\let\wh=\widehat

\def\dH{\dot{H}}

\def\cB{{\mathcal B}}
\def\cC{{\mathcal C}}

\def\cE{{\mathcal E}}
\def\cF{{\mathcal F}}

\def\cS{{\mathcal S}}

\def\virgp{\raise 2pt\hbox{,}}
\def\cdotpv{\raise 2pt\hbox{;}}

\def\eqdefa{\buildrel\hbox{\footnotesize def}\over =}

\def\C{\mathop{\mathbb C\kern 0pt}\nolimits}
\def\DD{\mathop{\mathbb D\kern 0pt}\nolimits}
\def\EE{\mathop{\mathbb E\kern 0pt}\nolimits}
\def\K{\mathop{\mathbb K\kern 0pt}\nolimits}
\def\N{\mathop{\mathbb  N\kern 0pt}\nolimits}
\def\Q{\mathop{\mathbb  Q\kern 0pt}\nolimits}
\def\R{{\mathop{\mathbb R\kern 0pt}\nolimits}}
\def\SS{\mathop{\mathbb  S\kern 0pt}\nolimits}
\def\St{\mathop{\mathbb  S\kern 0pt}\nolimits}
\def\Z{\mathop{\mathbb  Z\kern 0pt}\nolimits}
\def\ZZ{{\mathop{\mathbb  Z\kern 0pt}\nolimits}}
\def\H{{\mathop{{\mathbb  H\kern 0pt}}\nolimits}}
\def\PP{\mathop{\mathbb P\kern 0pt}\nolimits}
\def\TT{\mathop{\mathbb T\kern 0pt}\nolimits}

\def\h {{\rm h}}
\def\v {{\rm v}}
\def\na{\nabla}
\def\pa{\partial}

\newcommand{\ds}{\displaystyle}
\newcommand{\la}{\lambda}

\newcommand{\andf}{\quad\hbox{and}\quad}
\newcommand{\with}{\quad\hbox{with}\quad}

\def\dive{\mathop{\rm div}\nolimits}

\def\Supp{\mathop{\rm Supp}\nolimits\ }

\begin{document}

  \title[Some remarks about the possible  blow-up  for the Navier-Stokes equations] {Some remarks about the possible  blow-up  for the Navier-Stokes equations}

 \author[J.-Y. Chemin]{Jean-Yves  Chemin}
\address[J.-Y. Chemin]%
{Laboratoire Jacques Louis Lions - UMR 7598,  Sorbonne Universit\'e\\ Bo\^\i te courrier 187, 4 place Jussieu, 75252 Paris
Cedex 05, France}
\email{chemin@ann.jussieu.fr }
\author[I. Gallagher]{Isabelle Gallagher}
\address[I. Gallagher]%
{DMA, \'Ecole normale sup\'erieure, CNRS, PSL Research University, 75005 Paris
 \\
and UFR de math\'ematiques, Universit\'e Paris-Diderot, Sorbonne Paris-Cit\'e, 75013 Paris, France.}
\email{gallagher@math.ens.fr}
\author[P. ZHANG]{Ping Zhang}%
\address[P. Zhang]
 {Academy of
Mathematics $\&$ Systems Science and  Hua Loo-Keng Key Laboratory of
Mathematics, The Chinese Academy of Sciences, China, and School of Mathematical Sciences, University of Chinese Academy of Sciences, Beijing 100049, China.}
\email{zp@amss.ac.cn}
\subjclass[2010]{}
\keywords{}
\begin{abstract}
In this work we investigate the question of preventing  the three-dimensional, incompressible Navier-Stokes equations  from developing singularities, by controlling     one component of the velocity field only, in space-time scale invariant norms. In particular we prove that   it is not possible for one  component of the velocity field to tend to~$0$ too fast near blow up.
  We  also introduce a space ``almost" invariant under the action of the scaling  such that if one component of the velocity field measured in this space remains small enough, then there is no blow up.
  \end{abstract}

\maketitle

\noindent {\sl Keywords:}  Incompressible Navier-Stokes Equations,
Blow-up criteria, Anisotropic\\ Littlewood-Paley Theory\

\vskip 0.2cm
\noindent {\sl AMS Subject Classification (2000):} 35Q30, 76D03  \
\setcounter{equation}{0}

\section{Introduction}
The purpose of this paper is the investigation of the possible  behaviour of a solution of the incompressible Navier-Stokes equation in~$\R^3$ near the (possible) blow up time. Let us recall the form of the incompressible Navier-Stokes equation
\begin{equation*}
(NS)\qquad \left\{\begin{array}{l}
\displaystyle \pa_t v + \dive (v\otimes v) -\D v+\na p=0 \, , 
 \\
\displaystyle \dive\, v = 0 \andf  v|_{t=0}=v_0 \, ,
\end{array}\right.
\end{equation*}
where the unknowns~$v=(v^1,v^2, v^3)$ and~$p$  stand respectively for the velocity of the fluid and
   its pressure.

It is well known that the system has two main properties related to its physical origin:
\begin{itemize}
\item
the scaling invariance, which  states that if~$v(t,x)$ is a solution on~$[0,T]\times \R^3$ then for any positive real number~$\lambda$, the rescaled vector field~$v_\lam(t,x)\eqdefa \lam v(\lam^2 t,\lam x)$ is also a solution of $(NS)$ on~$[0,\lam^{-2} T]\times \R^3$;
\item
the dissipation of energy which writes
\beq
\label {energydissp}
\frac  12 \|v(t)\|_{L^2}^2 +\int_0^t \|\nabla v(t')\|_{L^2}^2 dt' \leq \frac 12 \|v_0\|_{L^2}^2\,.
\eeq
\end{itemize}

The first type of results which describe the behaviour of a (regular) solution  just before   blow up are those which are a consequence of an  existence theorem  for initial data in spaces more regular than the scaling. The seminal text~\ccite{lerayns} of J. Leray already pointed out in~1934 that the life span~$T^\star(v_0)$ of the regular solution  associated with an initial data in the Sobolev space~$H^1(\R^3)$  is greater than~$c\|\nabla v_0\|_{L^2}^{-4}$;   applying this result with~$v(t)$  as an initial data gives immediately that if~$T^\star(v_0)$ is finite, then
\beq
\label {lifespanLeray}
\|\nabla v(t)\|^4_{L^2} \geq \frac c {T^\star(v_0) -t} \, \virgp \quad\hbox {which implies that}\quad
\int_0^{T^\star(v_0)} \|\nabla v(t)\|_{L^2}^4 dt =\infty\,.
\eeq
More generally, it is a classical result that for any~$\gamma$ in~$]0,1/2[$ there holds
\beq
\label {minorlifeHs}
T^\star(v_0)\geq c_\g \|v_0\|_{\dot H^{\frac 12+2\gamma}}^{-\frac 1 \g} \quad\hbox {which leads to}\quad
\|v(t)\|_{\dot H^{\frac 12+2\gamma}} \geq \frac {c_\g} {(T^\star(v_0)-t)^\g} \,\cdotp
\eeq
The definition of homogeneous Sobolev spaces is recalled in the Appendix. Let us notice that the above formula is scaling invariant and comes from the resolution of~$(NS)$ with a fixed point argument following the Kato method ~\cite{kato} and \ccite {weissler}. Moreover, E. Poulon proved in\ccite{poulon} that if a regular initial data exists, the associate solution of which blows up at a finite time, then an initial data~$\underline v_0$ exists in the unit sphere of~$\dot H^{\frac 1 2+\g}$ such that
$$
T^\star(\underline v_0) = \inf \bigl\{T^\star(u_0)\,,\ u_0\in  \dot H^{\frac 1 2+\g}\,,\ \|u_0\|_{\dot H^{\frac 1 2+\g}}=1  \bigr\}\,.
$$
Assertion\refeq {minorlifeHs} can be generalized to the norm associated with the greatest space which is translation invariant, continuously included  the space of tempered distributions~$\cS'(\R^3)$ and whose norm has the same scaling as~$\dot H^{\frac 12+2\g}$. As pointed out by Y. Meyer in Lemma 9 of\ccite{meyerNSlivre}, this space is the Besov space~$\dot B^{ -1+2\g}_{\infty,\infty}$ which can be defined as the space of distributions such that
\beq
\label {definBesovIntro}
\|u\|_{\dot B^{ -1+2\g}_{\infty,\infty}} \eqdefa \sup_{t>0} t^{\frac 12 -\g} \|e^{t\D} u \|_{L^\infty}
\eeq
is finite. The generalization  of the bound given by\refeq  {minorlifeHs} to this norms is not difficult (see for instance Theorem~1.3 of\ccite {cgTMJ2018}). Let us recall the statement.
\begin{theo}[\cite{cgTMJ2018}]
\label {wpgnsBesov}
{\sl
For any~$\g$ in the interval~$]0,1/2[$, a constant~$c_\g$ exists such that for any  regular initial data~$v_0$, its life span~$T^\star(v_0)$  satisfies
\begin{equation}
\label{lowerboundfpThm}
T^\star(v_0) \geq    c_\gamma \|v_0\|_{\dot B^{-1+2\gamma}_{\infty,\infty}}^{-\frac 1\gamma}
\quad\hbox {which leads to}\quad
\|v(t)\|_{\dot B^{-1+2\gamma}_{\infty,\infty}} \geq \frac {c_\g} {(T^\star(v_0)-t)^\g} \,\cdotp
\end{equation}
}
\end{theo}
This result has an analogue as a global regularity result under a smallness condition,  which is the  Koch and Tataru  theorem (see\ccite  {kochtataru}) which claims that an  initial data which has a small norm in the space~$BMO^{-1} (\R^3)$ generates a global unique solution (which turns out to be as regular  as the initial data). The space~$BMO^{-1} (\R^3)$ is a very slightly larger space than~$\dot B^{-1}_{\infty, 2} (\R^3)$ defined by
\beq
\label {definBinfty2intro}
\|u\|_{\dot B^{-1}_{\infty,2}}^2 \eqdefa \int_0^\infty \|e^{t\D} u\|^2_{L^\infty} dt <\infty
\eeq
and very slightly smaller than the space~$\dot B^{-1}_{\infty,\infty}$. Let us notice that the classical spaces~$\dot H^{\frac 12}(\R^3)$ and~$L^3(\R^3)$ are continuously embedded in~$BMO^{-1} (\R^3)$.

Let us point out that the proof of all these results does not use the special structure of~$(NS)$ and in particular they  are true for any system of the type
$$
(GNS)\quad \partial_t v -\D v +\sum_{i,j} A_{i,j}(D) (v^iv^j)=0
$$
where~$A_{i,j}(D)$ are smooth homogenenous Fourier multipliers of order~$1$. The problem investigated here is to improve the description of the behavior of the solution near a possible blow up with the help  of the special structure of the non linear term of the Navier-Stokes equation.

One major achievement in this field is the work\ccite {ISS} by L. Escauriaza, G. Seregin and V. Sver\'ak  which proves that
\beq
\label {ISSeq}
T^\star (v_0)<\infty\Longrightarrow \limsup_{t\rightarrow T^\star (v_0) } \|v(t)\|_{L^3} =\infty\,.
\eeq
This was extended to the full limit in time  in~$\dot H^\frac12$ and not only the upper limit by G. Seregin in~\cite{seregin}\,.

 A different  context consists in formulating a condition which involves only one component of the velocity field.  The first result in that direction was obtained in a pioneering work by
J. Neustupa and P.  Penel (see\ccite{NP}) but the norm involved was
not scaling invariant.   A lot of works (see \cite{CT1, CT,H,  KZ,
NNP, PP, P, SK, ZP}) establish conditions of the type
$$
 \int_0^{T^\star} \|v^3(t,\cdot)\|^p_{L^q}dt=\infty \quad\hbox{or}\quad
 \int_0^{T^\star} \|\partial_j v^3(t,\cdot)\|^p_{L^q}dt=\infty
 $$
 with relations on~$p$ and~$q$ which however still fail to   make these quantities  scaling invariant.

The first result  in that direction involving a scaling invariant  condition was proved by the first and  the third author in\ccite {CZNSblowup}\,. It claims that  for any regular intial data with gradient in~$L^\frac 32 (\R^3)$ and  for any  unit vector~$\s\in{\mathbb S}^2,$  there holds
\beq
\label {czresulteq}
T^\star <\infty \Longrightarrow \int_0^{T^\star}\|v(t)\cdot \s\|_{\dH^{\frac12+\frac 2p}}^p\,dt=\infty\,,
\eeq
for any~$p$ in the interval~$]4,6[$. This was extended by Z. Zhang and the first and the third author to any~$p$ greater than~$4$ in\ccite {CZZ}\,. This is the analogue  of the integral condition of\refeq  {lifespanLeray} for  one  component only.

The first result of this paper is the analogue of\refeq {czresulteq} in the case when~$p=2$. More precisely, we prove the following theorem.
\begin{theo}
\label {czzp=2}
{\sl Let~$v$ be a maximal solution of~$(NS)$ in~$C([0,T^\star[;\dot H^1)$. If~$T^\star$ is finite, then  $$
\forall \s\in \SS^2\,,\ \int_0^{T^\star} \|v(t)\cdot \s\|_{\dot H^{\frac 32}}^2 dt =\infty\,.
$$}
\end{theo}
The proof we present here differs from the proofs in\ccite {CZNSblowup} and\ccite {CZZ}\,. It is simpler but seems to be specific to the case when~$p=2$.

\medbreak

The  rest of the paper is devoted to the study of what happens to the above criteria in the case when~$p$ is infinite. In other words,  is it  possible to extend the L. Escauriaza, G. Seregin and V. Sver\'ak  criterion\refeq {ISSeq}  for one component only? This question sems too ambitious for the time being. Indeed, following the work\ccite {gkp} by G. Koch, F. Planchon and the second author\footnote {See also\ccite{cheminPekin2014}   for a more elementary approach}   one way to understand  the work of L. Escauriaza, G. Seregin and V. Sver\'ak is the following: assume that a solution exists such that the $\dot H^{\frac 12}$ norm remains bounded near the blow up time. The first step consists in proving that the solution tends weakly to~$0$ when~$t$  tends to the blow up time. The second step consists in proving a backward uniqueness result which implies that the solution is~$0$, which of course contradicts the fact that it blows up in finite time.
The first step relies in particular on the fact that the Navier-Stokes system~$(NS)$ is globally wellposed for small data in~$\dot H^{\frac12}$. In our context, the equivalent statement would be that if~$\|v_0\cdot \s\|_{\dot H^{\frac 12}}$ is small enough for some unit vector~$\s$ of~$\R^3$, then there is a global regular solution. Such a result, assuming it is true, seems out of reach for the time being.

The result we prove in this paper is that if there is   blow up, then it is not possible for one  component of the velocity field to tend to~$0$ too fast. More precisely, we are going to prove the following theorem.
\begin{theo}
\label {blowupanisoscale0elem}
{\sl A  positive constant~$c_0$ exists such that for any initial data~$v_0$ in~$H^1(\R^3)$ with    associate solution~$v$ of
 $(NS)$ blowing  up at a finite time~$T^\star$, for any unit vector~$\sigma$ of~${\mathbb S}^2,$ there holds
$$
\forall t < T^\star \, , \quad \sup_{t'\in [t,T^\star[} \|v(t')\cdot \sigma\|_{\dot H^{\frac 12}} \geq c_0
\log^{-\frac12}  \Bigl( e+  \frac {\|v(t)\|_{L^2}^4}{ T^\star -t}\Bigr)\,\cdotp
$$}
\end{theo}

The other result we prove here  requires         reinforcing slightly the~$\dot H^{\frac 12}$ norm, while remaining (almost) scaling invariant.
\begin{defin}
\label {definH12log}
{\sl Let~$E$ be a positive real number and~$\s$ an element of the unit sphere~$\SS^2$.  We define~$\dot H^{\frac12}_{\log_\s,E}$ the space of distributions~$a$ in the homogeneous space~$\dot H^{\frac 12} (\R^3)$ such that
$$
\|a\|_{\dot H^{\frac12}_{\log_\s,E}}^2 \eqdefa \int_{\R^3} |\xi| \log\bigl(|\xi_\s| E +e\bigr) |\wh a(\xi)|^2 d\xi<\infty \with \xi_\s \eqdefa \xi-(\xi\cdot \s)\s \, .
$$
}
\end{defin}

Our theorem  is the following.
 \begin{theo}
\label {blowupanisoscale0elem2}
{\sl  A  positive constant~$c_0$ exists which satisfies the following. If~$v$
is
 a maximal solution of~$(NS)$  in~$C([0,T^\star[;  H^1)$ and if~$T^\star $ is finite, then  for any positive real number~$E$,
$$
\forall \sigma\in \SS^2\,,\quad \limsup_{t\rightarrow T^\star }   \|v(t)\cdot \sigma \|_{\dot H^\frac12_{\rm{log_\s},E}}\ \geq c_0\, .
$$}
\end{theo}

The structure of the paper is the following: in Section\refer {proofTheo12}, we reduce the proof of the three theorems to the proofs of three lemmas. The basic idea in this section consists in estimating the~$L^2$  norm of the horizontal derivatives of the solution.  Let us point out that the standard~$L^2$ energy estimate plays an important role. These ideas are common to the proof of the three theorems.

Section\refer  {prooflemmas} is devoted to the proof of the three lemmas. The one relative to Theorem\refer {blowupanisoscale0elem2} uses paradifferential calculus.

\section {Proof of the  theorems}
\label {proofTheo12}
Following an idea of\ccite{CT}, we perform an~$L^2$ scalar product   on the  momentum equation of~$(NS)$ with~$-\D_\h v $.
 This can be interpreted as a~$\dot H^1$ energy estimate for the horizontal variables. Recalling that
 $$\nabla_\h v\eqdefa ( \partial_1 v, \partial_2 v)\, ,
 $$
  we have
\beq
\label {blowupanisoscale0elemdemoeq1}
\begin{aligned}
\frac 12 \frac d {dt} \|\na_\h v\|_{L^2}^2 +\|\na_\h v\|_{\dot H^1}^2  &= \sum_{j=1}^4 \cE_j(v) \with\\
\cE_1(v) &\eqdefa -\sum_{i=1}^2\bigl(\partial_iv^\h\cdot\na_\h v^\h \big | \pa_iv^\h\bigr)_{L^2}\,,\\
\cE_2(v)  &\eqdefa -\sum_{i=1}^2\bigl(\partial_iv^\h\cdot\na_\h v^3 \big | \pa_iv^3\bigr)_{L^2}\,,\\
\cE_3(v) &\eqdefa -\sum_{i=1}^2(\partial_iv^3\pa_3v^\h \big | \pa_iv^\h)_{L^2}\andf\\
\cE_4(v) &\eqdefa -\sum_{i=1}^2(\partial_iv^3\pa_3v^3 \big | \pa_iv^3)_{L^2}\,.
\end{aligned}
\eeq
Let $\dive_\h v^\h\eqdefa \pa_1v^1+\pa_2v^2.$  A direct computation shows  that
\beno
\cE_1(v)=-\int_{\R^3}\dive_\h v^\h\Bigl(\sum_{i,j=1}^2(\pa_iv^j)^2+\pa_1v^2\pa_2v^1-\pa_1v^1\pa_2v^2\Bigr)\,dx\,,
\eeno
which, together with $\dive v=0$, ensure that
\beno
\cE_1(v)=\int_{\R^3}\pa_3 v^3\Bigl(\sum_{i,j=1}^2(\pa_iv^j)^2+\pa_ 1v^2\pa_2v^1-\pa_1v^1\pa_2v^2\Bigr)\,dx\,.
\eeno
Then let us observe that the three terms~$\cE_1(v)$,~$\cE_2(v)$ and~$\cE_4(v)$  are sums of terms of the form
\beq
\label {blowupanisoscale0elemdemoeq1bis}
I(v) \eqdefa \int_{\R^3} \partial_i v^3(x) \partial_jv^k(x) \partial_\ell v^m(x) dx
\eeq
with~$(j,\ell)$ in~Ê$\{1,2\}^2$ and~$(i,k,m)$ in~$\{1,2,3\}^3$.

Without loss of generality, we shall always take $\s=e_3$ in the rest of this paper.
\begin{proof}[Proof of Theorem{\rm\refer {czzp=2}}]
 H\"older's inequality implies that
$$
|I(v) |  \leq  \|\nabla v^3\|_{L^3} \|\nabla_\h v\|^2_{L^3}\,.
$$
The Sobolev embedding~$\dot H^{\frac 12} \hookrightarrow L^3$ and an interpolation inequality between~$L^2$ and~$\dot H^1$ imply
$$
|I(v) | \lesssim \|v^3\|_{\dot H^{\frac 32}}\|\nabla_\h v\|_{L^2 } \|\nabla _h v\|_{\dot H^1}\,.
$$
A convexity inequality then  gives
\beq
\label {proofTheo12eq1}
|I(v) | \leq \frac 1 {100}  \|\nabla _h v\|^2 _{\dot H^1}  +C \|v^3\|^2_{\dot H^{\frac 32}}\|\nabla_\h v\|^2_{L^2 }\,.
\eeq
In order to estimate~$\cE_3(v)$, we have to study   terms of the type
\beq
\label {defJ}
J_{i\ell}(v, v^3)  \eqdefa \int_{\R^3} \partial_iv^3\pa_3v^\ell \pa_iv^\ell \,dx,
\eeq
with~$(i,\ell)\in\{1,2\}^2$.
This is achieved through the following lemma, which will be proved in Section\refer {prooflemmas}\,.
\begin{lemme}
\label {inegloganisobasic}
A constant~$C$ exists such that, for any positive real number~$E$, we have
$$
\bigl |J_{i\ell}(v, v^3)  \bigr| \leq \frac 1 {10} \|\nabla_\h v\|_{\dot H^1}^2 +C\Bigl( \log\bigl(\|\nabla_\h v\|_{L^2}^2E+e\bigr)\|v^3\|^2_{\dot H^{\frac 32}} +  \frac  C {E} \|\nabla v\|_{L^2}^2 \Bigr)  \|\nabla_\h v\|_{L^2}^2 \,.
$$
\end{lemme}

\noindent {\it Continuation of the proof of Theorem{\rm\refer {czzp=2}}\,.\ } Using\refeq  {blowupanisoscale0elemdemoeq1},\refeq {proofTheo12eq1} and Lemma\refer   {inegloganisobasic}, we infer that
\beq
\label {democzzp=2eq1}
\begin{aligned}
&\frac 12 \frac d {dt} \|\na_\h v\|_{L^2}^2 +\|\na_\h v\|_{\dot H^1}^2 \leq \frac  12  \|\na_\h v\|_{\dot H^1}^2 \\
&\qquad\qquad\qquad\qquad{}+C\Bigl( \log\bigl(\|\nabla_\h v\|_{L^2}^2E+e\bigr)\|v^3\|^2_{\dot H^{\frac 32}} +  \frac  1 {E} \|\nabla v\|_{L^2}^2 \Bigr)  \|\nabla_\h v\|_{L^2}^2
\end{aligned}
\eeq
which implies that
$$
\frac d {dt} \log  \bigl(\|\nabla_\h v\|_{L^2}^2E+e\bigr) \lesssim  \|v^3\|_{\dot H^{\frac 32}}^2\log  \bigl(\|\nabla_\h v\|_{L^2}^2E+e\bigr)
+\frac  1 {E} \|\nabla v\|_{L^2}^2\,.
$$
Gronwall's lemma implies that
$$
\log  \bigl(\|\nabla_\h v(t)\|_{L^2}^2E+e\bigr) \leq  \bigl(\log  (\|\nabla_\h v_0\|_{L^2}^2E+e)+ \frac C E\int_0^t \|\nabla v(t')\|_{L^2}^2 dt'\bigr)
\exp  \Bigl( C  \int_0^t \|v^3(t')\|_{\dot H^{\frac 32}}^2 dt' \Bigr).
$$
The energy estimate~(\ref{energydissp}) then provides
$$
\log  \bigl(\|\nabla_\h v(t)\|_{L^2}^2E+e\bigr) \leq  \bigl(\log  (\|\nabla_\h v_0\|_{L^2}^2E+e)+ \frac {C\|v_0\|_{L^2}^2} E\bigr)
\exp  \Bigl( C  \int_0^t \|v^3(t')\|_{\dot H^{\frac 32}}^2 dt' \Bigr)\,\cdotp
$$
Thus if~$u$ is a~$C([0,T[;\dot H^1)$  solution of~$(NS)$ and if~$\ds \int_0^T  \|v^3(t)\|_{\dot H^{\frac 32}}^2 dt$ is finite, then~$\nabla_\h v$ is in~$L^\infty([0,T[; L^2)$. Plugging this in\refeq  {democzzp=2eq1}
implies also that~$\nabla_\h v$ is in~$L^2([0,T[; \dot H^1)$.

\medbreak

At this stage we can invoke Theorem 1.4 of \ccite{CZNSblowup} to conclude, because the  vertical component~$v^3$ of the solution remains bounded in the inhomogeneous space~$H^1$ on the time interval~$[0,T[$. For the reader's convenience, we present here a elementary and  self contained proof,   inspired by a method introduced in\ccite {cdgg2}\,. Differentiating~$(NS)$ with respect to the vertical variable and taking the~$L^2$ scalar product  of this system with~$\partial_3v$ gives, thaks to the divergence free condition
\beno
\frac 12 \frac  d {dt} \|\partial_3v(t)\|_{L^2}^2  +\|\partial_3 v(t)\|_{\dot H^1}^2
& =  & \sum_{k=1}^3 \biggl( \sum_{j=1}^2 \int_{\R^3} \partial_3 v^j \partial_j v^k\partial_3 v^k dx
+\int_{\R^3} \partial_3 v^3 \partial_3 v^k\partial_3v^k dx\biggr) \\
& = &   \sum_{k=1}^3 \biggl( \sum_{j=1}^2 \int_{\R^3} \partial_3 v^j \partial_j v^k\partial_3 v^k dx
-\int_{\R^3} \dive_\h v^\h \partial_3 v^k\partial_3v^k dx\biggr).
\eeno
All the terms on the right-hand side of the above equality can be estimated by
$$
\|\partial_3 v\|_{L^6} \|\nabla_\h v\|_{L^3} \|\partial_3 v\|_{L^2}
$$
Sobolev embeddings and an interpolation between~$L^2$ and~$\dot H^1$, along with    the convexity inequality imply that
$$
\frac 12 \frac  d {dt} \|\partial_3v(t)\|_{L^2}^2  +\|\partial_3 v(t)\|_{\dot H^1}^2
\leq \frac 12 \|\partial_3 v(t)\|_{\dot H^1}^2 + C \|\nabla_\h v(t)\|_{L^2} \|\nabla_\h v(t)\|_{\dot H^1}\|\partial_3v(t)\|_{L^2}^2\,.
$$
As we have
$$
\sup_{t\in [0,T[} \|\nabla_\h  v(t)\|_{L^2}^2 +\int_0^T \|\nabla_\h v(t)\|_{\dot H^1}^2 dt <\infty\,,
$$
the solution~$v$ remains bounded in~$\dot H^1$ and thus~$T$ cannot be the maximal time of existence.  Theorem\refer {czzp=2} is proved.
\end{proof}

\begin{proof} [Proof of Theorem{\rm\refer  {blowupanisoscale0elem}}]
We restart from\refeq {blowupanisoscale0elemdemoeq1} and\refeq{blowupanisoscale0elemdemoeq1bis}. Laws of product in three dimensional Sobolev spaces ensure that
\ben
\nonumber
I(v) & \leq & \|\partial_i v^3\|_{\dot H^{-\frac12}} \bigl \| \partial_jv^k \partial_\ell v^m\bigr\|_{\dot H^{\frac 12}} \\
\label {S2eq1}
&  \lesssim  & \|v^3\| _{\dot H^{\frac12}} \|\nabla_\h v\|_{\dot H^1}^2.
\een
Now let us turn to the estimate of~$J$ defined in~(\ref{defJ}).
Then the proof relies on the following lemma, which we shall prove in Section\refer {prooflemmas}.
\begin{lemme}
\label {Lemmavorticityanisoelem}
{\sl For any positive~$\e$, a  constant~$C_\e$ exists such that
for any  positive constant $E$  there holds
$$
\big |J_{i\ell}(v, v^3)\big |\leq  \Bigl(  \e +C_\e\|v^3\|_{\dot H^{\frac 1 2} }  \sqrt {\log\bigl( e+E\|\nabla_\h v\|_{L^2}^2\bigr)} \ \Bigr) \|\nabla_\h v\|^2_{\dot H^1}
+ C_\e \|v^3\|_{\dot H^{\frac 12}}^2    \frac{ \|\partial_3 v\|^2_{L^2} }{E^2}\,\cdotp
$$}
\end{lemme}

\noindent {\it Continuation of the proof of Theorem{\rm\refer  {blowupanisoscale0elem}}. \ }
Considering that~$\cE_3(v)$ is a sum of expressions of the type~Ê$J_{i\ell}(v,v^3)$, let us apply this lemma along with Inequality\refeq {S2eq1}. Plugging those results into\refeq {blowupanisoscale0elemdemoeq1}  gives
\beq
 \label {blowupanisoscale0elemdemoeq10}
 \begin{aligned}
\frac 12 \frac d {dt} \|\na_\h v\|_{L^2}^2 +\|\na_\h v\|_{\dot H^1}^2  &\leq  \Bigl(  \frac 14  +C\|v^3\|_{\dot H^{\frac 1 2} }  \sqrt {\log\bigl( e+E\|\nabla_\h v\|^2_{L^2}\bigr)} \ \Bigr) \|\nabla_\h v\|^2_{\dot H^1}\\
&\qquad\qquad\qquad\qquad\qquad\qquad \qquad\quad+
 C \|v^3\|^2_{\dot H^{\frac 12}}    \frac{ \|\partial_3 v\|^2_{L^2} }{E^2}\,\cdotp
\end{aligned}
\eeq
Let us define, for~$T\leq T^\star$,
\beq
\begin{aligned}
 \label {blowupanisoscale0elemdemoeq11}
m(T) &\eqdefa \sup_{t\in [0, T[} \|v^3(t)\|_{\dot H^{\frac 12}}\andf \\
 \underline T&\eqdefa \sup\bigl\{T'\leq T\,/\ Ê
\|\na_\h v\|_{L^\infty([0,T'];L^2)}^2\leq 2\|\na_\h v_0\|^2_{L^2}\bigr\} \,.
\end{aligned}
\eeq
Let us note that for any divergence free vector field in~$\dot H^1$, we have
$$
\|\na w^3\|_{L^2}^2  =  \|\na_\h w^3\|_{L^2}^2+\|\pa_3w^3\|_{L^2}^2
 =  \|\na_\h w^3\|_{L^2}^2+\|\dive_\h w^\h\|_{L^2}^2
 $$
 which ensures that
 \beq
 \label {majodiv0gradw3}
 \|\na w^3\|_{L^2}^2 \leq 2\|\nabla_\h w\|_{L^2}^2\,.
 \eeq
 Then  for $t\leq \underline T,$ Inequality\refeq {blowupanisoscale0elemdemoeq10} becomes
 $$
 \longformule{
\frac 12 \frac d {dt} \|\na_\h v\|_{L^2}^2 +\|\na_\h v\|_{\dot H^1}^2  \leq  \Bigl(  \frac 14  +C_0 m(T)  \sqrt {\log\bigl( e+E\|\nabla_\h v_0\|^2_{L^2}\bigr)} \ \Bigr) \|\nabla_\h v\|^2_{\dot H^1}\\
}
{ {}
+
 C m^2 (T)   \frac{ \|\partial_3 v\|^2_{L^2} }{E^2}\,\cdotp
 }
 $$
 Let us choose~$E$ equal to ~$\frac{\|v_0\|_{L^2}}{\|\na_h v_0\|_{L^2}}$ and let us assume that
 \beq
 \label {smallnessinduction}
 m(T)  \sqrt {\log\bigl( e+\|v_0\|_{L^2}\|\nabla_\h v_0\|_{L^2}\bigr)} \leq c_0
 \eeq
 with small enough~$c_0$ (less than~$1/4C_0$ for the time being). Then we get
$$
 \frac d {dt} \|\na_\h v\|_{L^2}^2 +\|\na_\h v\|_{\dot H^1}^2  \lesssim m^2(T)\|\nabla_\h v_0\|^2_{L^2}  \frac{ \|\partial_3 v(t)\|^2_{L^2} }{\|v_0\|_{L^2}^2}\,\cdotp
$$
By time integration and using the energy inequality \eqref{energydissp}, we infer that for any~Ê$t\leq \underline T$,
\beq
\label{S2eq3a}
\|\na_\h v(t)\|_{L^2}^2 \leq \|\na_\h v_0\|_{L^2}^2 \bigl( 1+C_1 m^2(T)\bigr).
\eeq
Then the argument used at the end of the proof of Theorem\refer  {czzp=2} can be repeated. So by contraposition, we infer that
$$
m(T^\star)  \sqrt {\log\bigl( e+\|v_0\|_{L^2}\|\nabla_\h v_0\|_{L^2}\bigr)} \geq c_0\,.
$$
Now let us translate in time  this assertion. Defining
$$
M(t) \eqdefa \sup_{t'\in [t,T^\star]} \|v^3(t')\|_{\dot H^{\frac 12}},
$$
the above assertion claims that
$$
M(t)  \geq \frac {c_0} { \sqrt {\log\bigl( e+\|v(t)\|_{L^2}\|\nabla_\h v(t)\|_{L^2}\bigr)}} \,\cdotp
$$
If~$M(t)$ is infinite there is nothing to prove. Let us assume~$M(t)$ (which is a non-increasing function)  is finite. Then the above inequality can be written
$$
\forall t'\in [t,T^\star]\,,\ \|v(t')\|^2_{L^2}\|\nabla_\h v(t')\|^2_{L^2} \geq \exp \Bigl( \frac {2c_0^2} {M^2(t')} \Bigr)\,\cdotp
$$
Because of the decay of the kinetic energy and the monotonicity of~$M$, this can be written
$$
\forall t'\in [t,T^\star]\,,\ \|v(t)\|^2_{L^2}\|\nabla_\h v(t')\|^2_{L^2} \geq \exp \Bigl( \frac {2c_0^2} {M^2(t)} \Bigr)\,\cdotp
$$
By integration of this inequality in the interval~$[t,T^\star]$, we infer that
$$
\|v(t)\|^2_{L^2}\int_t^{T^\star} \|\nabla_\h v(t')\|^2_{L^2}  dt' \geq  (T^\star-t) \exp \Bigl( \frac {2c_0^2} {M^2(t)} \Bigr)\,\cdotp
$$
The energy estimate  implies that
$$
\frac 12 \|v(t)\|_{L^2}^4 \geq  (T^\star-t) \exp \Bigl( \frac {2c_0^2} {M^2(t)} \Bigr)
$$
which concludes the proof of Theorem\refer   {blowupanisoscale0elem}\,.
\end{proof}

\begin{proof}  [Proof of Theorem{\rm\refer  {blowupanisoscale0elem2}}]  Using\refeq {blowupanisoscale0elemdemoeq1},\refeq  {blowupanisoscale0elemdemoeq1bis} and\refeq {S2eq1}, we get, recalling notation~(\ref{defJ}),
\beq
\label{estimateThm1.4}
\frac 12 \frac d {dt} \|\na_\h v\|_{L^2}^2 +\|\na_\h v\|_{\dot H^1}^2 \lesssim \|v^3\|_{\dot H^{\frac 12}}\|\nabla_\h v\|_{\dot H^1}^2 +\sum_{(i,\ell) \in \{1,2\}^2} \big| J_{i\ell}(v,v^3) \big| \, .
\eeq
Then the proof relies on the following lemma.
\begin{lemme}
\label {Lemmavorticityanisoelem2}
{\sl  Let us define
$$
\|a\|^2_{\dot H^{\frac 12}_{\log_\h,E}} \eqdefa \int_{\R^3} |\xi|\, | \wh a(\xi)|^2 \log (|\xi_\h|E+e) d\xi\,.
$$
A constant~$C$ exists such that, for any positive~Ê$E$, we have
\beq \label{S4eq3}
\big| J_{i\ell}(v,v^3) \big| \leq  \Bigl (\frac 1 {10} +C  \|v^3\|_{\dot H^{\frac 12}_{\rm{log}_\h,E}} \Bigr) \|\na_\h v^\h\|_{\dot H^1}^2 + C\|v^3\|_{\dot H^{\frac 12}}^2 \frac{  \|\partial_3v^\h \|_{L^2}^{2}}{E ^2}\,\cdotp \eeq
}
\end{lemme}

\noindent {\it Conclusion of the proof of Theorem{\rm\refer  {blowupanisoscale0elem2}}. \ } Let us plug this lemma  into\refeq{estimateThm1.4}. This gives
\beq
 \label {blowupanisoscale0elemdemoeq102}
\frac 12 \frac d {dt} \|\na_\h v(t)\|_{L^2}^2 +\|\na_\h v\|_{\dot H^1}^2 \leq \Bigl (\frac 1 {4} +C  \|v^3\|_{\dot H^{\frac 12}_{\rm{log}_\h,E}} \Bigr) \|\na_\h v^\h\|_{\dot H^1}^2 + C\|v^3\|_{\dot H^{\frac 12}}^2 \frac{  \|\partial_3v^\h \|_{L^2}^{2}}{E ^2}\,\cdotp
\eeq
Then by time integration and thanks to the energy estimate we find that as long as
$$
t\leq T_\ast\eqdefa \sup \Bigl\{ \, T\in ]0, T^\star[\,/\  \sup_{t\in [0,T]}\|v^3(t)\|_{\dot H^{\frac 12}_{\rm log_\h,E}}  \leq \frac1{4C}\ \ \Bigr\}\, ,
$$
there holds
\beno
 \|\na_\h v(t)\|_{L^2}^2 + \int_0^ t\|\na_\h v(t')\|_{\dot H^1}^2 \, dt' &\leq& \|\na_\h v_0\|_{L^2}^2 + \frac{1 }{2E^2}\int_0^t\|\pa_3v^\h(t')\|_{L^2}^2\,dt'\\
 &\leq& \|\na_\h v_0\|_{L^2}^2 + \frac{\|v_0\|_{L^2}^2 }{E^2}\, \cdotp
\eeno
Then to conclude we use the same arguments as in the conclusion of the previous two theorems. Theorem \refer  {blowupanisoscale0elem2} is proved.
\end{proof}


\section {Proof of   the three  lemmas
}
\label {prooflemmas}
In this section we prove Lemmas~\ref{inegloganisobasic},~\ref{Lemmavorticityanisoelem}  and\refer{Lemmavorticityanisoelem2}\,.
We shall use the following notation: $L^p_\h$ will denote the space~$L^p(\R^2)$ in the horizontal variables~$x_\h\eqdefa (x_1,x_2)$ (and we shall write~$\R^2_\h$ to specify the space~$x_\h$ belongs to), and $L^p_\v$ will denote the space~$L^p (\R)$ in the vertical variable (and we shall write~$\R_\v$ to specify the space~$x_3$ belongs to).

The main problem   is that the control of~$\nabla_\h v$ in~$L^2_\h$ does not imply any control on~$v$, simply because the Sobolev space~$\dot H^1(\R^2_\h)$ is not  continuously included in~$\cS'(\R^2_\h)$. In order to overcome this problem, the idea consists in  decomposing~$v$ into a term containing only low horizontal frequencies, a term  containing only medium horizontal frequencies and a term containing only high  horizontal  frequencies, and in estimating each of those three terms differently.  More precisely, for a couple of positive real numbers~$(\lam,\Lam)$ such that~$\lam\leq \Lam$,  let us define
\beq
\label {aligneda}
\begin{aligned}
a_{\flat, \lam} \eqdefa \cF^{-1} ( {\bf 1}_{B_\h (0,\lam)}& \wh a)\,,\ a_{\natural,\lam,\Lam}\eqdefa  \cF^{-1} \bigl(( {\bf 1}_{B_\h (0,\Lam)}
- {\bf 1}_{B_\h (0,\lam)})\wh a\bigr )  \andf\\
 & a_{\sharp, \Lam}  \eqdefa \cF^{-1}({\bf 1}_{B^c_\h (0,\Lam)}\wh a)   \,.
  \end{aligned}
\eeq
\begin{proof}[Proof of Lemma{\rm\refer {inegloganisobasic}}]

Let us study first low  horizontal frequencies. We start by writing
\beno
J^\flat_\lam & \eqdefa & \Bigl | \int_{\R^3} \partial_3 v_{\flat,\lam} ^\ell \partial_i v^3\partial_i v^\ell dx\Bigr|\\
 &\leq & \|\partial_3 v_{\flat,\lam} ^\ell \|_{L^4_\h(L^2_\v)} \|\partial_i v^3\|_{L^4_\h(L^2_\v)}  \|\partial_i v^\ell\|_{L^2_\h(L^\infty_\v)}\,.
\eeno
Using Bernstein and Gagliardo-Nirenberg inequalities in the horizontal variable (see the Appendix for anisotropic Bernstein inequalities), we infer that
\beno
J^\flat_\lam &\lesssim & \lam^{\frac 12}  \|\partial_3 v ^\ell \|_{L^2} \|\partial_i v^3\|_{L^2}^{\frac 12}
\|\partial_i \nabla_\h v^3\|_{L^2}^{\frac 12} \|\partial_i v^\ell\|_{L^2}^{\frac 12} \|\partial_3\partial_i v^\ell\|_{L^2}^{\frac 12}\\
& \lesssim &  \lam^{\frac 12}  \|\partial_3 v ^\ell \|_{L^2}\|\nabla_\h v\|_{L^2}\|\nabla_\h v\|_{\dot H^1} \,.
\eeno
By convexity we infer that
\beq
\label {inegloganisobasicdemoeq1}
J^\flat_\lam \leq \frac 1 {100} \|\nabla_\h v\|_{\dot H^1}^2 +C \lam \|\partial_3 v ^\ell \|^2_{L^2}\|\nabla_\h v\|^2_{L^2}\,.
\eeq
Let us now estimate the high horizontal frequency term. Using  Bernstein (inverse) and Ga\-gliardo-Nirenberg inequalities in the horizontal variable, we infer that
\ben
\nonumber
J^\sharp_\Lam & \eqdefa & \Bigl | \int_{\R^3} \partial_3 v_{\sharp,\Lam} ^\ell \partial_i v^3\partial_i v^\ell dx\Bigr|\\
\nonumber
&\leq & \|\partial_3 v_{\sharp,\Lam} ^\ell \|_{L^4_\h(L^2_\v)} \|\partial_i v^3\|_{L^4_\h(L^2_\v)}  \|\partial_i v^\ell\|_{L^2_\h(L^\infty_\v)}\\
\nonumber
 &\lesssim & \Lam^{-\frac 12}  \|\nabla_\h \partial_3 v ^\ell \|_{L^2} \|\partial_i v^3\|_{L^2}^{\frac 12}
\|\partial_i \nabla_\h v^3\|_{L^2}^{\frac 12} \|\partial_i v^\ell\|_{L^2}^{\frac 12} \|\partial_3\partial_i v^\ell\|_{L^2}^{\frac 12}\\
\label {inegloganisobasicdemoeq2} & \lesssim &  \Lam^{-\frac 12}  \|\nabla_\h v\|_{L^2}\|\nabla_\h v\|^2_{\dot H^1} \,.
\een
In order to estimate the term concerning the medium horizontal frequencies, let us write that
\ben
\nonumber
J^\natural_{\lam,\Lam} & \eqdefa & \Bigl | \int_{\R^3} \partial_3 v_{\natural,\lam,\Lam} ^\ell \partial_i v^3\partial_i v^\ell dx\Bigr|\\
\label {inegloganisobasicdemoeq3}
&\leq & \|\partial_3 v_{\natural,\lam,\Lam} ^\ell \|_{L^2_\v(L^\infty_\h)} \|\partial_i v^3\|_{L^\infty_\v(L^2_\h)}  \|\partial_i v^\ell\|_{L^2}\,.
\een
Using Bernstein  inequalities in the horizontal variables we get, for any~$x_3$ in~$\R$,
\beno
\|\partial_3 v_{\natural,\lam,\Lam} ^\ell (\cdot,x_3) \|_{L^\infty_\h} & \leq &  \sum_{\lam\lesssim 2^k \lesssim \Lam}
\|\partial_3 \D_k^\h v ^\ell  (\cdot,x_3)\|_{L^\infty_\h}\\
& \leq &  \sum_{\lam\lesssim 2^k \lesssim \Lam}
2^k \|\partial_3 \D_k^\h v ^\ell  (\cdot,x_3) \|_{L^2_\h}\,.
\eeno
By definition of the Sobolev norm in terms of Littlewood-Paley theory (see the Appendix),  we infer
$$
2^k \|\partial_3 \D_k^\h v ^\ell  (\cdot,x_3) \|_{L^2_\h}  \lesssim c_k(x_3) \|\nabla_\h \partial_3 v(\cdot,x_3) \|_{L^2_\h} \with \sum_k
c^2_k(x_3)=1\,.
$$
 Then using the Cauchy-Schwarz inequality, we infer that
\beno
\|\partial_3 v_{\natural,\lam,\Lam} ^\ell (\cdot,x_3) \|_{L^\infty_\h} & \lesssim & \|\nabla_\h \partial_3 v (\cdot,x_3)\|_{L^2_\h} \sum_{\lam\lesssim 2^k \lesssim \Lam} c_k(x_3)\\
& \lesssim & \sqrt {\log \Bigl(\frac \Lam \lam\Bigr)} \,\|\nabla_\h \partial_3 v (\cdot,x_3)\|_{L^2_\h}\,.
\eeno
Taking the~$L^2$ norm with respect to the variable~$x_3$ gives
\beq
\label {inegloganisobasicdemoeq4}
\|\partial_3 v_{\natural,\lam,\Lam} ^\ell \|_{L^2_\v(L^\infty_\h)}  \lesssim  \sqrt {\log \Bigl(\frac \Lam \lam\Bigr)} \|\nabla_\h v \|_{\dot H^1} \,.
\eeq
Now let us  observe that thanks to Lemma~\ref{LinftyHs}
$$
 \|\partial_i v^3\|_{L^\infty_\v(L^2_\h)} \leq \|(-\Delta_\h)^\frac14v^3\|_{ L^\infty_\v(\dot H^\frac12_\h)} \lesssim  \|v^3\|_{\dot H^{\frac 32}}\, .
$$
Plugging this inequality and\refeq {inegloganisobasicdemoeq4} into\refeq  {inegloganisobasicdemoeq3}
gives
$$
J^\natural_{\lam,\Lam}   \lesssim  \sqrt {\log \Bigl(\frac \Lam \lam\Bigr)} \|\nabla_\h v \|_{\dot H^1} \|v^3\|_{\dot H^{\frac 32}} \|\nabla_\h v \|_{L^2}\,.
$$
By convexity we get
$$
J^\natural_{\lam,\Lam}  \leq \frac 1 {100} \|\nabla_\h v \|_{\dot H^1} ^2 +C \log \Bigl(\frac \Lam \lam\Bigr)
 \|v^3\|^2_{\dot H^{\frac 32}} \|\nabla_\h v \|^2_{L^2}\,.
$$
Together with\refeq {inegloganisobasicdemoeq1} and\refeq {inegloganisobasicdemoeq2}, we infer that, for any couple of positive real numbers~$\lam$ and~$\Lam$ such that~$\lam$ is less than~$\Lam$, there holds
$$
\longformule{
\Bigl | \int_{\R^3} \partial_3 v^\ell \partial_i v^3\partial_i v^\ell dx \Bigr| \leq \frac 1 {50} \|\nabla_\h v\|^2_{\dot H^1}
+ C\Lam^{-\frac 12}  \|\nabla_\h v\|_{L^2}\|\nabla_\h v\|^2_{\dot H^1}
}
{
{}+ C\Bigl( \lam \|\nabla v\|^2_{L^2} + \|v^3\|^2_{\dot H^{\frac 32}} \log \Bigl( \frac \Lam \lam\Bigr)\Bigr) \|\nabla_\h v\|_{L^2}^2\,.
}
$$
Choosing
\beq
\label {inegloganisobasicdemoeq4b}
\Lam =(50 C)^2 \|\nabla_\h v\|_{L^2}^2 +\frac e E \andf \lam =\frac 1 E
\eeq
ensures the result.
\end{proof}

\begin{proof}[Proof of Lemma{\rm\refer {Lemmavorticityanisoelem}}] Let us focus on the estimate of the term~$J^\natural_{\lam,\Lam} $.
Applying Lemma~\ref{productlawh}, we infer that
\beno
J^\natural_{\lam,\Lam} & = & \int_{\R_\v} \Bigl (\int_{\R^2_\h} \partial_i v^3(x_h,x_3) 
\partial_3v^\ell_{\natural,\lam,\Lam} (x_h,x_3)  \partial_iv^\ell(x_h,x_3) dx_\h\Bigr) dx_3\\
& \leq &  \int_{\R_\v} \|\partial_i v^3(\cdot,x_3)\|_{\dot H^{-\frac 12}_\h}  \bigl \|
\partial_3v^\ell_{\natural,\lam,\Lam} (\cdot,x_3)  \partial_iv^\ell(\cdot,x_3)\|_{\dot H^{\frac 12}_\h} dx_3\\
& \lesssim &  \int_{\R_\v} \|\partial_i v^3(\cdot,x_3)\|_{\dot H^{-\frac 12}_\h}  \bigl (\|
\partial_3v^\ell_{\natural,\lam,\Lam} (\cdot,x_3) \|_{L^\infty_\h} +\|
\partial_3v^\ell_{\natural,\lam,\Lam} (\cdot,x_3)\|_{\dot H^1_\h} \bigr)\|\partial_iv^\ell(\cdot,x_3)\|_{\dot H^{\frac 12}_\h} dx_3\\
& \lesssim &  \|\nabla_\h v\|_{L^\infty_\v(\dot H^{\frac 12}_\h)}\int_{\R_\v} \|\partial_i v^3(\cdot,x_3)\|_{\dot H^{-\frac 12}_\h}  \bigl (\|
\partial_3v^\ell_{\natural,\lam,\Lam} (\cdot,x_3) \|_{L^\infty_\h} +\|
\partial_3v^\ell_{\natural,\lam,\Lam} (\cdot,x_3)\|_{\dot H^1_\h} \bigr) dx_3\,.
\eeno
As we obviously have
$$
\|\partial_3v^\ell_{\natural,\lam,\Lam} \|_{L^2_\v(\dot H^1_\h)} \leq \|\nabla _\h v\|_{\dot H^1}\andf
\|\partial_i v^3\|_{L^2_\v(\dot H^{-\frac 12}_\h)} \leq \|v^3\|_{\dot H^{\frac 12}} \,,
$$
 we infer, using\refeq  {inegloganisobasicdemoeq4}  and the Cauchy-Schwarz inequality, that
\beq\label{S3eq6}
J^\natural_{\lam,\Lam}  \lesssim   \|\nabla_\h v\|_{L^\infty_\v(\dot H^{\frac 12}_\h)}
 \|\nabla _\h v\|_{\dot H^1} \|v^3\|_{\dot H^{\frac 12}} \sqrt {\log \Bigl(\frac \Lam \lam\Bigr)}\,\cdotp
\eeq
Let us observe  that thanks to Lemma~\ref{LinftyHs} there holds
$$ \|\nabla_\h v\|_{L^\infty_\v(\dot H^{\frac 12}_\h)} \lesssim  \|\nabla_\h v\|_{\dot H^1}\,.
$$
Along the same lines, we get, by using the product law~$(\dot B^1_{2,1})_\h \times \dH^{\frac12}_\h\subset \dH^{\frac12}_\h$ (see the appendix for the definition of~$(\dot B^1_{2,1})_\h$), that
\beno
J^\flat_\lam & = & \Bigl | \int_{\R^3} \partial_3 v_{\flat,\lam} ^\ell \partial_i v^3\partial_i v^\ell dx\Bigr|\\
 &\lesssim & \|v^3\|_{\dH^{\frac12}}\|\partial_3 v_{\flat,\lam} ^\ell \partial_i v^\ell\|_{L^2_\v(\dH^{\frac12}_\h)}\\
 &\lesssim & \|v^3\|_{\dH^{\frac12}}\|\partial_3 v_{\flat,\lam}\|_{L^2_\v(\dot B^1_{2,1})_\h}\| \partial_i v^\ell\|_{L^\infty_\v(\dH^{\frac12}_\h)} \, .
\eeno
But according to the definition of $v_{\flat,\lam},$ there holds
\beno
\|\partial_3 v_{\flat,\lam}\|_{L^2_\v(\dot B^1_{2,1})_\h}&\leq &\sum_{2^k\leq \la}2^k\|\D_k^\h \partial_3v_{\flat,\lam}\|_{L^2}\\
&\lesssim& \la \|\partial_3 v\|_{L^2} \, ,
\eeno
so
\beq\label{S3eq7}
\begin{aligned}
J^\flat_\lam &\leq C\la \|v^3\|_{\dH^{\frac12}}\|\partial_3 v\|_{L^2}\|\na_\h v\|_{\dH^1}\\
&\leq  \frac1{100}\|\na_\h v\|_{\dH^1}^2+C\la^2 \|v^3\|_{\dH^{\frac12}}^2\|\partial_3 v\|_{L^2}^2 \, .
\end{aligned}
\eeq
Then   with the choice of~$\lam$ and~$\Lam$ made in\refeq {inegloganisobasicdemoeq4b}, the lemma is proved using \refeq {inegloganisobasicdemoeq2}, \eqref{S3eq6} and~\eqref{S3eq7}.
\end{proof}

\begin{proof}  [Proof of Lemma{\rm\refer {Lemmavorticityanisoelem2}}] The main point consists in estimating
$$
J^\sharp_{E} \eqdefa \int_{\R^3} \partial_iv^3 \partial_3v^\ell_{\sharp,E^{-1}} \partial_i v^\ell dx\,.
$$
Using Bony's decomposition in the horizontal variable  introduced in the proof of Lemma\refer {productlawh}, let us write that
\beq
\label {Lemmavorticityanisoelem2demoeq1}
\begin{aligned}
J^\sharp_{E} &= J^{\sharp,1} _{E} +J^{\sharp,2} _{E} \with\\
J^{\sharp,1} _{E} & \eqdefa \int_{\R_\v}\biggl(  \int_{\R^2_h}\partial_i v^3 (x_h, x_3)
\wt T^\h_{\partial_i v^\ell(\cdot,x_3)} \partial_3v^\ell_{\sharp,E^{-1}} (\cdot ,x_3) dx_\h\biggr) dx_3\andf\\
J^{\sharp,2} _{E} & \eqdefa \int_{\R_\v}\biggl( \sum_{k\in \ZZ} \int_{\R^2_h} \D_k^\h \partial_i v^3(\cdot , x_3)
\wt \D_k^\h T^\h_{\partial_3v^\ell_{\sharp,E^{-1}}(\cdot,x_3)} \partial_i v^\ell (\cdot ,x_3) dx_\h\biggr) dx_3\,.
\end{aligned}
\eeq
Let us estimate~$J^{\sharp,1} _{E}$. The control of this term does not use the fact that
~$v^\ell_{\sharp,E^{-1}}$ contains only high horizontal frequencies.  Using\refeq {productlawhdemoeq1}, we can write
\beno
J^{\sharp,1} _{E} & \leq & \int_{\R_\v} \|\nabla_\h v^3 (\cdot, x_3) \|_{\dot H^{-\frac 12}(\R^2_\h)}
\| \wt T^\h_{\partial_i v^\ell(\cdot,x_3)}\partial_3v^\ell_{\sharp,E^{-1}} (\cdot ,x_3) \|_{\dot H^{\frac 12}(\R^2_\h)}  dx_3\\
& \lesssim & \|\nabla_\h v\|_{L^\infty_\v(\dot H^{\frac 12}(\R^2_\h))}
\int_{\R_\v} \|v^3 (\cdot, x_3) \|_{\dot H^{\frac 12}(\R^2_\h)}
\|\partial_3v^\ell (\cdot ,x_3) \|_{\dot H^1(\R^2_\h)}  dx_3\,.
\eeno
Using  the Cauchy-Schwarz inequality gives
\beno
J^{\sharp,1} _{E} & \lesssim & \|\nabla_\h v\|_{L^\infty_\v(\dot H^{\frac 12}(\R^2_\h))}
\|v^3 \|_{L^2_\v(\dot H^{\frac 12}(\R^2_\h))}
\|\partial_3v^\ell \|_{L^2_\v(\dot H^1(\R^2_\h))}\\
& \lesssim & \|\nabla_\h v\|_{L^\infty_\v(\dot H^{\frac 12}(\R^2_\h))}
\|v^3 \|_{\dot H^{\frac 12}(\R^3)}
\|\nabla_\h v^\ell \|_{\dot H^1(\R^3)}\,.
\eeno
Once observed that (see Lemma~\ref{LinftyHs})~$\|\nabla_\h v\|_{L^\infty_\v(\dot H^{\frac 12}(\R^2_\h))} \lesssim \|\nabla_\h v\|_{\dot H^1(\R^3)}$, we infer that
\beq
\label {Lemmavorticityanisoelem2demoeq2}
J^{\sharp,1} _{E} \lesssim \|v^3 \|_{\dot H^{\frac 12}(\R^3)}
\|\nabla_\h v\|^2_{\dot H^1(\R^3)}\,.
\eeq
In order to estimate~$J^{\sharp,2}_E$, we revisit the proof of Theorem~2.47 in\ccite {bcdbookk} which describes the mapping of paraproducts. Because the support of the Fourier transform of
$$
S_{k'-1}^\h(\partial_3v^\ell_{\sharp,E^{-1}}(\cdot,x_3)) \D_{k'}^\h \partial_i v^\ell (\cdot ,x_3)
$$
is included in a ring of~$\R^2_\h$ of the type~$2^{k'} \wt \cC$, we get, by definition of~$T^\h$  that
$$
J^{\sharp,2} _{E} = \int_{\R_\v}\Biggl( \sum_{k\in \ZZ} \int_{\R^2_h} \D_k^\h \partial_i v^3(\cdot , x_3)
\wt \D_k^\h \Bigl( \sum_{|k'-k|\leq N_0} S_{k'-1}^\h(\partial_3v^\ell_{\sharp,E^{-1}}(\cdot,x_3)) \D_{k'}^\h \partial_i v^\ell (\cdot ,x_3)\Bigr)  dx_\h\Biggr) dx_3\,.
$$
Using Cauchy-Schwarz and Bernstein inequalities, we get
$$
\longformule{
J^{\sharp,2} _{E}  \leq   \int_{\R_\v}\biggl(\sum_{|k-k'|\leq N_0} 2^{k} \|\D_k^\h v^3(\cdot,x_3)\|_{L^2_\h}
}
{
{}\times
 \|S_{k'-1}^\h\partial_3v^\ell_{\sharp,E^{-1}}(\cdot,x_3)\|_{L^\infty_\h}
\| \D_{k'}^\h \partial_i v^\ell (\cdot ,x_3)\|_{L^2_\h} \biggr)dx_3\,.
}
$$
Using the fact that~$|k'-k|\leq N_0$ we get, using the equivalence\refeq {definSovnormLP} of    Sobolev norms,
$$
\longformule{
J^{\sharp,2} _{E}  \lesssim \| \nabla _\h v\|_{L^\infty_\v(\dot H^{\frac 12}(\R^2_\h)} \int_{\R_\v}  \biggl(\sum_{k\in \ZZ} 2^{\frac k 2} \|\D_k^\h v^3(\cdot,x_3)\|_{L^2_\h} c_k(x_3)
}
{
{} \times  \sum_{|k'-k| \leq N_0}
 \|S_{k'-1}^\h\partial_3v^\ell_{\sharp,E^{-1}}(\cdot,x_3)\|_{L^\infty_\h} \biggr)dx_3 \with\sum_{k\in \ZZ} c_k^2(x_3) =1\, .
 }
$$
Thanks to Bernstein's inequality and by the equivalence\refeq {definSovnormLP} of    Sobolev norms, we can  write 
\beno
\|S_{k'-1}^\h\partial_3v^\ell_{\sharp,E^{-1}}(\cdot,x_3)\|_{L^\infty_\h}  & \leq & \sum_{E^{-1} \lesssim 2^{k''} \leq 2^{k'-2}} \|\D_{k''}^\h\partial_3v^\ell_{\sharp,E^{-1}}(\cdot,x_3)\|_{L^\infty_\h} \\
& \lesssim & \sum_{E^{-1} \lesssim 2^{k''} \leq 2^{k'-2}} 2^{k''}\|\D_{k''}^\h\partial_3v^\ell_{\sharp,E^{-1}}(\cdot,x_3)\|_{L^2_\h} \\
& \lesssim & \sum_{E^{-1} \lesssim 2^{k''} \leq 2^{k'-2}} 2^{k''}\|\D_{k''}^\h\partial_3v^\ell(\cdot,x_3)\|_{L^2_\h} \\
& \lesssim & \|\nabla_\h \partial_3v^\ell(\cdot,x_3)\|_{L^2_\h} \sum_{E^{-1} \lesssim 2^{k''} \leq 2^{k'-2}}  c_{k''}(x_3)
\eeno
with~$\ds \sum_{k\in \ZZ} c_{k}^2(x_3) =1.$ Because~$|k'-k|\leq N_0$,  the Cauchy-Schwarz inequality implies the existence of a constant~$C$ such that
$$
 \forall x_3\in \R_\v\,,\ \sum_{E^{-1} \lesssim 2^{k''} \leq 2^{k'-2}}  c_{k''}(x_3) \leq C \sqrt {\log (2^kE+e)} \,.
$$
Thus we infer that
$$
\longformule{
J^{\sharp,2} _{E}  \lesssim \| \nabla _\h v\|_{L^\infty_\v(\dot H^{\frac 12}(\R^2_\h))} \int_{\R_\v}  \|\nabla_\h \partial_3v^\ell(\cdot,x_3)\|_{L^2_\h}  }
{
{} \times \biggl(\sum_{k\in \ZZ} 2^{\frac k 2} \|\D_k^\h v^3(\cdot,x_3)\|_{L^2_\h} c_k(x_3)\sqrt {\log (2^kE+e)}
 \biggr)dx_3 \with\sum_{k\in \ZZ} c_k^2(x_3) =1.
 }
$$
Using    the Cauchy-Schwarz inequality we get
$$
\sum_{k\in \ZZ} 2^{\frac k 2} \|\D_k^\h v^3(\cdot,x_3)\|_{L^2_\h} c_k(x_3)\sqrt {\log (2^kE+e)}
\leq \biggl(\sum_{k\in \ZZ} 2^{k } \|\D_k^\h v^3(\cdot,x_3)\|^2_{L^2_\h} \log (2^kE+e)\biggr)^{\frac 12}.
$$
Using    the Cauchy-Schwarz inequality again we infer that
$$
J^{\sharp,2} _{E}  \lesssim \| \nabla _\h v\|_{L^\infty_\v(\dot H^{\frac 12}(\R^2_\h))}
\|\nabla_\h v \|_{\dot H^1(\R^3)} \|v^3\|_{\dot H^{\frac 12}_{\log_\h,E}}\,.
$$
Lemma~\ref{LinftyHs} claims that
$$
\|\nabla_\h v\|_{L^\infty_\v(\dot H^{\frac 12}(\R^2_\h))} \lesssim \|\nabla_\h v\|_{\dot H^1(\R^3)}\, ;
$$
so we obtain
\beno
J^{\sharp,2} _{E}  \lesssim \| \nabla _\h v\|^2_{\dot H^1(\R^3)}
 \|v^3\|_{\dot H^{\frac 12}_{\log_h,E}}  \,.
\eeno
Together with\refeq {S3eq7} and \refeq  {Lemmavorticityanisoelem2demoeq2}, this concludes the proof of Lemma\refer {Lemmavorticityanisoelem2}\,.
\end{proof}

\appendix

\section{}\label{apB}
Let us recall some elements of  Littlewood-Paley theory (see for instance\ccite {bcdbookk} for   details). We define frequency truncation operators
 on~$\R^2$,
\begin{equation}
\label{defparaproduct}
\Delta_k^{\rm h}a\eqdefa \cF^{-1}(\varphi(2^{-k}|\xi_{\rm h}|)\widehat{a})\andf S^{\rm h}_ka\eqdefa\cF^{-1}(\chi(2^{-k}|\xi_{\rm h}|)\widehat{a})\, ,
\end{equation}
where $\xi=(\xi_\h, \xi_3)$ and $\xi_{\rm h}=(\xi_1,\xi_2)$. We have denoted $\cF a$ and
$\widehat{a}$ for the Fourier transform of the distribution~$a,$ and~$\chi  $ and~$\varphi$ are smooth functions such that
 \beno
&&\Supp \varphi \subset \Bigl\{\tau \in \R\,/\  \ \frac34 \leq
|\tau| \leq \frac83 \Bigr\}\andf \  \ \forall
 \tau>0\,,\ \sum_{j\in\Z}\varphi(2^{-j}\tau)=1\,,\\
&&\Supp \chi \subset \Bigl\{\tau \in \R\,/\  \ \ |\tau|  \leq
\frac43 \Bigr\}\quad \ \ \ \andf \  \ \forall
 \tau\in\R\,,\  \chi(\tau)+ \sum_{j\geq
0}\varphi(2^{-j}\tau)=1\,.
 \eeno
It is obvious that  \beq
\label {definSovnormLP}
\|a\|_{\dot H^s_\h}\eqdefa\|a\|_{\dot H^s(\R^2)} \eqdefa\bigl \|
|\cdot|^s\widehat a
\bigl\|_{L^2(\R^2)} \sim \bigl\| ( 2^{ks} \|\D^\h_k a\|_{L^2(\R^2)} )\bigr\|_{\ell^2(\ZZ)}\,.
\eeq
We recall the definition of horizontal Besov norms
$$
\|a\|_{(\dot B^s_{p,q})_\h} \eqdefa \|a\|_{\dot B^s_{p,q}(\R^2)} \eqdefa  \bigl\| ( 2^{ks} \|\D^\h_k a\|_{L^p(\R^2)} )\bigr\|_{\ell^q(\ZZ)} \, .
$$
We also recall the following anisotropic
Bernstein type lemma from \cite{CZ1, Pa02}\,.
\begin{lemme}
\label{lemBern}
{\sl Consider $\cB_{\h}$  a ball
of~$\R^2_{\h}$ and~$\cC_{\h}$ a
ring of~$\R^2_{\h}$ ; fix~$1\leq p_2\leq p_1\leq
\infty$. Then the following properties hold:

\smallbreak\noindent -  If the support of~$\wh a$ is included
in~$2^k\cB_{\h}$, then
\[
\|\partial_{x_{\rm h}}^\alpha a\|_{L^{p_1}_{\rm h}}
\lesssim 2^{k\left(|\al|+2\left(1/{p_2}-1/{p_1}\right)\right)}
\|a\|_{L^{p_2}_{\rm h}}\,.
\]
- If the support of~$\wh a$ is included in~$2^k\cC_{\h}$, then
\[
\|a\|_{L^{p_1}_{\rm h}} \lesssim
2^{-k} \|\nabla_\h a\|_{L^{p_1}_{\rm
h}}\,.
\]
}
\end{lemme}
The following lemma is also useful in this anisotropic context.
\begin{lemme}
\label{LinftyHs}
{\sl
For any function~$a$ in the space~$\dot H^{s+\frac12}(\R^3)$ with~$1/2\leq s<1$, there holds
$$
\|a\|_{L^\infty_\v(\dot H^s_\h)} \leq \sqrt 2 \|a\|_{ \dot H^{s+\frac12}(\R^3)}\,.
$$
}
\end{lemme}
\begin{proof}
By density we can assume that~$v$ is smooth and compactly supported. Let us write that
$$
\frac d {dx_3} \int_{\R^2_\h}    |\xi_\h|^{2s} |\wh a(\xi_\h, x_3)|^2 \, d\xi_\h  = 2
\Re e \int_{\R^2_\h}  |\xi_\h|^{s+\frac 12}  \wh a(\xi_\h, x_3)
|\xi_\h|^{s-\frac 12}\partial_{x_3}\wh a(\xi_\h, x_3) \,d\xi_\h\,.
$$
The Cauchy-Schwarz inequality implies that
$$
\Bigl | \frac d {dx_3} \int_{\R^2_\h}  |\xi_\h|^{2s} |\wh a(\xi_\h, x_3)|^2
\, d\xi_\h \Bigr| \leq  2
\|a(\cdot,x_3)\|_{\dot H_\h^{s+\frac 12} } \|\partial_{x_3} a(\cdot,x_3)\|_{\dot
H_\h^{s-\frac 12} } \,.
$$
Taking the~$L^1$ norm with respect to~$x_3$ and using again the Cauchy-Schwarz
inequality gives, for any~$x_3$ in~$\R_\v$,
$$
\int_{\R^2_\h}  |\xi_\h|^{2s} |\wh a(\xi_\h, x_3)|^2 \, d\xi_\h  \leq 2
\|a\|_{L^2_\v(\dot H_\h^{s+\frac 12})} \|\partial_{x_3} a \|_{L^2_\v(\dot
H_\h^{s-\frac 12})} \leq 2\|a\|^2_{\dot H^{s+\frac 12}(\R^3)}\,.
$$
The lemma follows.
\end{proof}

We recall that
$$
\|ab\|_{\dot H^{\frac 12}(\R^2)} \lesssim   \|a\|_{\dot B^1_{2,1}(\R^2)}\|b\|_{\dot H^{\frac 12}(\R^2)} \, .
$$
The following law of product in~$\R^2$ is also useful.
\begin{lemme}
\label {productlawh}
{\sl
For any functions~$a\in L^\infty\cap\dot H^1 (\R^2)$ and~$b\in \dot H^{\frac 12}(\R^2)$, there holds
$$
\|ab\|_{\dot H^{\frac 12}(\R^2)} \lesssim \bigl( \|a\|_{L^\infty(\R^2)}  +\|a\|_{\dot H^1(\R^2)}\bigr) \|b\|_{\dot H^{\frac 12}(\R^2)}\,.
$$
}
\end{lemme}

\begin{proof}We recall Bernstein's inequality
$$
\|\Delta^\h_k  a\|_{L^\infty(\R^2)}+ \|\nabla _\h \Delta^\h_k  a\|_{L^2(\R^2)} \lesssim 2^k\|\Delta^\h_k  a\|_{L^2(\R^2)}\,.
$$
Let us introduce Bony's decomposition ( in a simplified version of \cite{Bo81}) writing that
\beq
\label {bonydecomph}
ab = T^\h_a b +\wt T^\h_b a \with T^\h_a b \eqdefa \sum_k S^\h_{k-1} a \D^\h_k b\andf
\wt T^\h_b a \eqdefa \sum_k S^\h_{k+2} b\D^\h_k a\,.
\eeq
Theorem~2.47 and Theorem~2.52 of\ccite {bcdbookk}  claim that
\beq
\label {productlawhdemoeq1}
\|T^\h_a b \|_{\dot H^{\frac 12} (\R^2)} \lesssim \|a\|_{L^\infty(\R^2)} \|b\|_{\dot H^{\frac 12}(\R^2)}\andf
\|\wt T^\h_b a \|_{\dot H^{\frac 12} (\R^2)} \lesssim \|a\|_{\dot H^1(\R^2)} \|b\|_{\dot H^{\frac 12}(\R^2)}\,.
\eeq
It is clear that\refeq {bonydecomph} and\refeq {productlawhdemoeq1} imply the lemma.
\end{proof}

\bigbreak \noindent {\bf Acknowledgments.}  Part of this work was done when Jean-Yves Chemin was visiting Morningside Center of Mathematics, 
Academy of mathematics and System Sciences. He would like to thank the hospitality of the center and  a visiting fellowship from  the Chinese
Academy of Sciences (CAS).   P. Zhang is partially supported
    by NSF of China under Grants   11371347 and 11688101,  and innovation grant from National Center for
    Mathematics and Interdisciplinary Sciences, CAS.

\end{document}